\documentclass[12pt, a4paper]{article}

\usepackage{a4}
\usepackage{latexsym}
\usepackage{amssymb}
\usepackage{amsfonts}
\usepackage{amsthm}
\usepackage{amsmath}
\usepackage[pdftex]{graphicx}
\usepackage{cite}
\usepackage{cases}

\newcommand{\essinf}[1]{\underset{#1}{\mbox{ess inf}} \,}

\newcommand{\eps}{{\varepsilon}}

\newcommand{\Xmax}{\overline{X}_{\infty}}
\newcommand{\XmaxT}{\overline{X}_{t}}

\newcommand{\bq}{\begin{equation}}
\newcommand{\eq}{\end{equation}}

\newtheorem{theorem}{Theorem}
\newtheorem{cor}[theorem]{Corollary}
\newtheorem{prop}[theorem]{Proposition}
\newtheorem{lemma}[theorem]{Lemma}

\newtheorem{example}[theorem]{Example}
\newtheorem{remark}[theorem]{Remark}

\renewcommand{\theequation}{\arabic{section}.\arabic{equation}}

\begin{document}
\title{Predicting the time at which a L\'evy process attains its ultimate supremum}
\author{Erik Baurdoux\footnote{Department of Statistics, London School of Economics and Political Science. Houghton street, {\sc London, WC2A 2AE, United Kingdom.} E-mail: e.j.baurdoux@lse.ac.uk} \quad \& \quad Kees van Schaik\footnote{School of Mathematics, University of Manchester. Oxford Road, {\sc Manchester, M13 9PL, United Kingdom.} E-mail: kees.vanschaik@manchester.ac.uk}}
\date{\today}
\maketitle

\begin{abstract}
\noindent
We consider the problem of finding a stopping time that minimises the $L^1$-distance to $\theta$, the time at which a L\'evy process attains its ultimate supremum. This problem was studied in \cite{DuToit08} for a Brownian motion with drift and a finite time horizon. We consider a general L\'evy process and an infinite time horizon (only compound Poisson processes are excluded. Furthermore due to the infinite horizon the problem is interesting only when the L\'evy process drifts to $-\infty$). Existing results allow us to rewrite the problem as a classic optimal stopping problem, i.e. with an adapted payoff process. We show the following. If $\theta$ has infinite mean there exists no stopping time with a finite $L^1$-distance to $\theta$, whereas if $\theta$ has finite mean it is either optimal to stop immediately or to stop when the process reflected in its supremum exceeds a positive level, depending on whether the median of the law of the ultimate supremum equals zero or is positive. Furthermore, pasting properties are derived. Finally, the result is made more explicit in terms of scale functions in the case when the L\'evy process has no positive jumps.
\end{abstract}

\noindent
{\footnotesize Keywords: L{\'{e}}vy processes, optimal prediction, optimal stopping.}

\noindent
{\footnotesize Mathematics Subject Classification (2000): 60G40, 62M20}

\section{Introduction}\label{sec_intro}
\setcounter{equation}{0}

This paper addresses the question of how to predict the time a L\'evy process attains its ultimate supremum with an infinite time horizon. (Due to the jumps a L\'evy process can experience, the word ``attains'' is used here in a slightly broader sense than when the driving process is continuous, cf. Section \ref{sec_main1}). That is, we aim to find a stopping time that is closest (in $L^1$ sense) to the time the L\'evy process attains its ultimate supremum. This is an example of an \emph{optimal prediction problem}. It is related to classic and well-studied optimal stopping problems. However, the key difference is that the payoff process is not adapted to the filtration generated by the driving stochastic process. Indeed, in our case, the time the L\'evy process attains its ultimate supremum is not known (with absolute certainty) at any (finite) time $t$. However, as time progresses, more information about the time of the ultimate supremum becomes available. Examples of optimal prediction problems where this is not the case include the ``hidden target'' type problems studied in \cite{Peskir10}.

In recent years optimal prediction problems have received considerable attention, see e.g. \cite{Graversen01, Cohen10, DuToit08, Urusov05, Peskir10, Allaart10, Glover11, Bernyk11, DuToit09, Espinosa10, DuToit07}. One reason is that such problems have found applications in fields such as engineering, finance and medicine. Prominent examples concern the optimal time to sell an asset (in finance) or the optimal time to administer a drug (in medicine).

The papers referred to above are mainly concerned with optimal prediction problems driven by Brownian motion (with drift), particularly with a finite time horizon. Some exceptions are \cite{Allaart10} which deals with random walks, \cite{Espinosa10} with mean-reverting diffusions and \cite{Bernyk11} which deals with spectrally positive stable processes. The same problem as we consider was studied in \cite{DuToit08}, however for a Brownian motion with drift and a finite time horizon. In that paper it is also postulated that for a general L\'evy process the structure of the solution is the same as they found for the Brownian motion with drift. See also Remark \ref{rem_Kevin}.

In this paper, the driving process is a general L\'evy process $X$ drifting to $-\infty$, i.e. $\lim_{t \to \infty} X_t=-\infty$ a.s. (otherwise the problem we consider is trivial as we will briefly point out in the sequel). We are interested in solving

\begin{equation}\label{def_main1}
\inf_{\tau} \mathbb{E}[|\theta-\tau|],
\end{equation}
where $\theta$ is the time $X$ attains its ultimate supremum (cf. Section \ref{sec_main1} for details) and the infimum is taken over all stopping times $\tau$ with respect to the filtration generated by $X$. Following \cite{DuToit08,Urusov05}, due to the stationarity and independence of the increments of $X$, (\ref{def_main1}) can be expressed as an optimal stopping problem driven by the reflected process $Y$ given by $Y_t=\overline{X}_t-X_t$ for all $t \geq 0$, where $\overline{X}_t = \sup_{0\leq s \leq t} X_s$ denotes the running supremum of $X$. This allows us to show the following. If $\theta$ has infinite mean then (\ref{def_main1}) is degenerate in the sense that it equals $\infty$. Now suppose $\theta$ has finite mean. If the law of the ultimate supremum $\overline{X}_{\infty} = \lim_{t \to \infty} \overline{X}_t$ has an atom in $0$ of size at least $1/2$ then $\tau=0$ is optimal in (\ref{def_main1}). Otherwise, the infimum in (\ref{def_main1}) is attained by the first time $Y$ enters an interval $[y^*,\infty)$ for some $y^*$ strictly larger than the median of the law of $\overline{X}_{\infty}$. We derive pasting properties and in the special case that $X$ is spectrally negative we obtain (semi-)explicit expressions for (\ref{def_main1}) and $y^*$ in terms of so-called scale functions.

Note that the problem (\ref{def_main1}) with a finite time horizon (rather than an infinite time horizon as we consider) is potentially more interesting from the point of view of applications. However, it is also considerably more challenging than the infinite horizon case. As motivated in \cite{DuToit08} (cf. also Remark \ref{rem_Kevin}) it should be expected that in the finite horizon case the optimal stopping time is of the form the first time that the reflected process $Y$ exceeds some time-dependent curve. For a Brownian motion with drift as considered in \cite{DuToit08}, this curve is specified only implicitly as the solution to a nonlinear integral equation. Their proof relies on first reducing the optimal prediction problem to an optimal stopping problem (also with a  finite horizon, cf. their Lemma 2), which is then solved by the well-known technique of representing the value function of the optimal stopping problem and the time-dependent curve determining the optimal stopping time as a system with a PDE in a domain that has a free boundary.

Hence, if for a Brownian motion with drift in the finite horizon case the optimal stopping time can only be implicitly characterised, it should not be expected to be anything more explicit when a more general L\'evy process drives the problem. As outlined above, in the infinite horizon case we derive rather explicit results, in particular for spectrally one-sided L\'evy processes. (However, our results should also lead to rather explicit results for the large family of meromorphic L\'evy processes (cf. \cite{Kuznetsov10}) for instance). In the finite horizon case, Lemma 2 from \cite{DuToit08} still applies and then any of the available approximation techniques for finite expiry optimal stopping problems driven by L\'evy processes could be used. See for instance \cite{Kleinert13}, \cite{levendorskii2006}, \cite{Maller06} and \cite{matache2005}. For a Canadisation based method (cf. \cite{Kleinert13} and the references therein) the results in this paper should prove useful as that method consists of solving a recursive sequence of optimal stopping problems, each of which is a variation of the infinite horizon problem.

The rest of this paper is organised as follows. In Section \ref{sec_prel} we discuss some preliminaries and some technicalities to be used later on. In Section \ref{sec_main1} we introduce our main result, Theorem \ref{thm_main}. Section \ref{sec_proof} is then dedicated to the proof of Theorem \ref{thm_main}. In Section \ref{sec_infmean} we discuss the case that $\theta$ has infinite mean. We make our results more explicit in the case that $X$ is spectrally negative in Section \ref{subsec_sn}. In that section we also work out the examples of a Brownian motion with drift, a Brownian motion with drift plus negative jumps and a drift plus negative jumps. Finally, in the Appendix we collect the proofs of the lemmas discussed in Section \ref{sec_prel}.

\section{Preliminaries}\label{sec_prel}
\setcounter{equation}{0}

Let $X=(X_t)_{t \geq 0}$ be a L\'evy process starting from $0$ defined on a filtered probability space $(\Omega,\mathcal{F},\mathbf{F},\mathbb{P})$, where $\mathbf{F}=(\mathcal{F}_t)_{t \geq 0}$ is the filtration generated by $X$ which is naturally enlarged (cf. Definition 1.3.38 in \cite{Bichteler02}). Recall that a L\'evy process is characterised by stationary, independent increments and paths which are right continuous and have left limits, and its law is characterised by the characteristic exponent $\Psi$ defined through $\mathbb{E}[e^{\mathrm{i} z X_t}]=e^{-t \Psi(z)}$ for all $t \geq 0$ and $z \in \mathbb{R}$. According to the L\'evy--Khintchine formula there exist $\sigma \geq 0$, $a \in \mathbb{R}$ and a measure $\Pi$ (the L\'evy measure) concentrated on $\mathbb{R} \setminus \{0\}$ satisfying $\int_{\mathbb{R}} (1 \wedge x^2) \, \Pi(\mathrm{d}x)<\infty$ (the tuple $(\sigma,a,\Pi)$ is usually refered to as the \emph{L\'evy triplet}) such that

\[ \Psi(z) = \frac{\sigma^2}{2} z^2 + \mathrm{i} az + \int_{\mathbb{R}} \left( 1-e^{\mathrm{i}zx}+\mathbf{1}_{\{ |x|<1 \}} \mathrm{i} zx \right) \, \Pi(\mathrm{d}x) \]
for all $z \in \mathbb{R}$. For further details see e.g. the textbooks \cite{Kyprianou06, Bertoin96, Sato99}.

We denote the running supremum at time $t$ by $\overline{X}_t = \sup_{s \leq t} X_s$ for all $t \geq 0$, so that $\overline{X}_{\infty} = \lim_{t \to \infty} \overline{X}_t$ is the ultimate supremum of $X$. As is well known (see e.g. Theorem 12 on p.167 in \cite{Bertoin96}), we have

\begin{equation}\label{18mei1}
\overline{X}_{\infty}<\infty \text{ a.s.} \quad \Leftrightarrow \quad  \lim_{t \to \infty} X_t = -\infty \text{ a.s.} \quad \Leftrightarrow \quad \int_1^{\infty} \frac{1}{s} \mathbb{P}(X_s \geq 0) \, \mathrm{d}s<\infty.
\end{equation}

In Section \ref{sec_main1}, we look at the problem of predicting the time of the ultimate supremum of $X$, which is defined as

\[ \theta = \inf \{ t \geq 0 \, | \, \overline{X}_t=\overline{X}_{\infty} \} \]
(cf. the discussion in Section \ref{sec_main1}). If $X$ is not compound Poisson then we have

\begin{equation}\label{21jun1}
\mathbb{E}[\theta]<\infty \quad \Leftrightarrow \quad \int_0^{\infty} \mathbb{P}(X_s \geq 0) \, \mathrm{d}s<\infty.
\end{equation}
Indeed, making use of the Sparre--Andersen identity (cf. Lemma 15 on p. 170 in \cite{Bertoin96})

\[ \int_0^{\infty} \mathbb{P}(X_s \geq 0) \, \mathrm{d}s = \lim_{t \to \infty} \mathbb{E} \left[ \int_0^{t} \mathbf{1}_{\{ X_s \geq 0 \}} \, \mathrm{d}s \right] = \mathbb{E} \left[ \lim_{t \to \infty} \sup \{ s<t \, | \, X_s=\overline{X}_s \} \right] = \mathbb{E}[\theta], \]
where the last equality holds since the set $\{t \geq 0 \, | \, X_t = \Xmax\ \text{ or } X_{t-} = \Xmax\}$ is a singleton when $X$ is not compound Poisson, see e.g. p. 158 of \cite{Kyprianou06}.

We shall soon see that the presence of an atom at $0$ in the law of $\Xmax$ plays a prominent role. Such an atom is related to \emph{(ir)regularity upwards}. In fact, by denoting

\begin{equation}\label{def_stoptime1}
\tau^+(x) = \inf\{ t>0 \, | \, X_t>x \} \quad \text{and} \quad \tau^-(x)=\inf \{ t>0 \, | \, X_t<x \},
\end{equation}
$X$ is said to be regular upwards if $\tau^+(0)=0$ a.s. -- i.e. if $X$ enters $(0,\infty)$ immediately after starting from $0$; otherwise (then $\tau^+(0)>0$ a.s.) $X$ is said to be irregular upwards. Similarly, $X$ is said to be regular (resp. irregular) downwards if $-X$ is regular (resp. irregular) upwards. Theorem 6.5 in \cite{Kyprianou06} classifies regularity upwards in terms of the L\'evy triplet. It is a well-known rule of thumb that the value function of an optimal stopping problem driven by $X$ exhibits so-called smooth or continuous pasting dependent on whether this property holds or not, see e.g. \cite{Kyprianou05} and the references therein (see also the main result of this paper, Theorem \ref{thm_main}).
We state four lemmas which will be of help to us to optimally predict the maximum of a L\'evy process. The proofs of these lemmas can be found in the Appendix.

The following lemma concerns the connection between an atom at $0$ of the law of $\Xmax$ and (ir)regularity upwards.

\begin{lemma}\label{Ffacts}
Suppose $X$ is not a compound Poisson process and drifts to $-\infty$. The law of $\overline{X}_{\infty}$ has an atom in $0$ if and only if $X$ is irregular upwards.
\end{lemma}

The next lemma allows us to conclude that when $X$ is not compound Poisson, the atom in $0$ identified in the above Lemma \ref{Ffacts} is the only possible atom in the law of $\Xmax$. Recall that $X$ is said to \emph{creep upwards} if for some (and then all) $x>0$ it holds that $\mathbb{P}(X_{\tau^+(x)}~=~x,\tau^+(x)~<~\infty)>0$. For instance, all L\'evy processes with a Gaussian component and those of bounded variation with a positive drift creep upwards, see e.g. Theorem 7.11 in \cite{Kyprianou06}.

The next lemma concerns continuity properties of the distribution function of $\Xmax$.

\begin{lemma}\label{lem_Fcont} Suppose $X$ is not a compound Poisson process (still drifting to $-\infty$). The distribution function $F$ of $\Xmax$ is continuous on $\mathbb{R}_{\geq 0}$. Furthermore, if $X$ creeps upwards then $F$ is Lipschitz continuous on $\mathbb{R}_{\geq 0}$.
\end{lemma}

  Note that the above result is not sharp: there are obvious examples of L\'evy processes not creeping upwards for which $F$ is nevertheless Lipschitz on $\mathbb{R}_{\geq 0}$, for instance when $X$ is a compound Poisson process with positive, exponentially distributed jumps plus a negative drift. In this case, when $\mathbb{E}[X_1]<0$ so that $\Xmax<\infty$ a.s. it holds $\mathbb{P}(\Xmax=0)>0$ (by Lemma \ref{Ffacts} above) while $\Xmax$ has a positive, bounded and continuous density on $\mathbb{R}_{>0}$ (cf. Theorem 2 in \cite{Mordecki02b}). For an interesting study of the law of the supremum of a L\'evy process we refer to \cite{Chaumont12}.

\begin{remark} Some examples of L\'evy processes with two-sided jumps for which the density of $\Xmax$ is known (semi-)explicitly are L\'evy processes with arbitrary negative jumps and phase-type positive jumps (cf. \cite{Mordecki02b}) and the class of so-called meromorphic L\'evy processes which have jumps consisting of a possibly infinite mixture of exponentials (cf. \cite{Kuznetsov10}). If $X$ has no positive jumps it is well known that $\Xmax$ follows an exponential distribution (cf. Section \ref{subsec_sn}), while if $X$ has no negative jumps, scale functions may be used to describe the law of $\Xmax$ (cf. \cite{Kyprianou06}).
\end{remark}

We conclude with two technical results that will be of use later.

\begin{lemma}\label{limsuph}
Recall the notation (\ref{def_stoptime1}). Suppose that $X$ is regular downwards, then
for any $c>0$
\[\limsup_{\eps\downarrow 0}\frac{\mathbb{P}(\tau^+(c-\eps)<\tau^-(-\eps))}{\eps}>0.\]
\end{lemma}

\begin{lemma}\label{lemmm} Let $X$ be any L\'evy process drifting to $-\infty$. Denote $T_+(0)=\inf \{ t \geq 0 \, | \, X_t \geq 0 \}$. Furthermore, for any $x \in \mathbb{R}$ denote by $\mathbb{P}_x$ (resp. $\mathbb{E}_x$) the measure (resp. expectation operator) under which $X$ starts from $x$. Consider for $a>0$ and $b<0$ the optimal stopping problem

\[ f(x) = \inf_{\tau} \mathbb{E}_x \left[ a \tau + \mathbf{1}_{\{ \tau \geq T_+(0) \}} b \right] \quad \text{for $x \in \mathbb{R}$}. \]
Then there is an $x_0 \in (-\infty,0)$ so that $f(x)=0$ for all $x \leq x_0$.
\end{lemma}

\section{Predicting the time of the ultimate supremum}\label{sec_main1}
\setcounter{equation}{0}

Define the time where the ultimate supremum of the L\'evy process $X$ is (first) attained by $\theta$, that is

\[ \theta = \inf \{ t \geq 0 \, | \, \XmaxT = \Xmax \}, \]
where we understand $\inf \emptyset=\infty$. Note that ``attained" is used in a loose sense here. Indeed, if $X$ has negative jumps it might happen that the ultimate supremum is never attained. However, the above definition ensures that we have $X_{\theta}=\Xmax$ on the event $\{ X_{\theta} \geq X_{\theta-} \}$ while $X_{\theta-}=\Xmax$ on the event $\{ X_{\theta}<X_{\theta-} \}$. Furthermore, when $X$ is not a compound Poisson process, the set $\{t \geq 0 \, | \, X_t = \Xmax\ \text{ or } X_{t-} = \Xmax\}$ is a singleton as already mentioned in Section \ref{sec_prel}.

Our aim in this section is to find a stopping time as close as possible to $\theta$, that is, we consider the optimal stopping problem
\begin{equation}\label{time_sup_def}
\inf_{\tau} \mathbb{E}[|\theta-\tau|],
\end{equation}
where the infimum is taken over all $\mathbf{F}$-stopping times.

The trichotomy of a L\'evy process at infinity (see Theorem 7.1 in \cite{Kyprianou06}) states that either $\lim_{t \to \infty} X_t = -\infty$, $\lim_{t \to \infty} X_t = \infty$ or $\limsup_{t \to \infty} X_t = -\liminf_{t \to \infty} X_t=\infty$ a.s. In the latter two cases we have $\theta=\infty$ a.s. and hence (\ref{time_sup_def}) is degenerate. Henceforth in this section we assume that

\begin{equation} \text{$X$ drifts to $-\infty$ and $\theta$ has finite mean}\label{henceforthassumption} \end{equation}
(we will deal with the case that $\theta$ has infinite mean in Proposition \ref{prop_theta_inf}). Recall that these properties were discussed in Section \ref{sec_prel}.

Note that as $\mathbb{E}[\theta]<\infty$ and $\mathbb{E}[|\theta-\tau|]\geq |\mathbb{E}[\theta]-\mathbb{E}[\tau]|,$ this standing assumption implies that without loss of generality we can  consider the infimum in (\ref{time_sup_def}) over all stopping times with finite mean.

 Recall that we denote by $F$ the distribution function of $\Xmax$. We introduce for any $y \geq 0$ the process $Y^y$, which is $X$ reflected in its running supremum, started from $y$, that is

\[ Y^y_t = (y \vee \overline{X}_t)-X_t \quad \text{for all $t \geq 0$.} \]
Note that for any $y \geq 0$, $Y^y$ is a strong Markov process which drifts to $\infty$ (as $-X$ does so). Later on, we shall also be using the following stopping times for $x,y \geq 0$:

\begin{equation}\label{def_stopTimes}
\sigma(y,x)=\inf \{ t \geq 0 \, | \, Y^y_t \geq x \} \quad \text{and} \quad \sigma_+(y)=\inf \{ t>0 \, | \, Y^y_t >y \}.
\end{equation}

Following \cite{DuToit08} Lemma 2 and \cite{Urusov05} Lemma 1, we rewrite the expectation in (\ref{time_sup_def}) as an expectation involving an $\mathbf{F}$-adapted process which will allow us to apply standard optimal stopping techniques. We include the proof here for completeness.

\begin{prop}\label{prop_repr} For any stopping time $\tau$ with finite mean we have that
\[ \mathbb{E}[|\theta-\tau|]=2\mathbb{E} \left[ \int_0^{\tau}  F(Y_t^0)\mathrm{d}t \right] + \mathbb{E}[\theta]-\mathbb{E}[\tau] = \mathbb{E} \left[ \int_0^{\tau} \left( 2 F(Y_t^0)-1 \right) \, \mathrm{d}t \right] + \mathbb{E}[\theta]. \]
\end{prop}

\begin{proof}

We have
\begin{eqnarray*}
|\theta-\tau|&=&\theta-\tau+ 2(\tau-\theta) \mathbf{1}_{\{\theta\leq \tau\}}\\
&=&\theta-\tau+2 \int_0^\tau \mathbf{1}_{\{\theta\leq t\}} \, \mathrm{d}t.
\end{eqnarray*}
Applying Fubini's Theorem twice we deduce that
\begin{eqnarray*}
\mathbb{E}\left[\int_0^\tau \mathbf{1}_{\{\theta\leq t\}}\, \mathrm{d}t\right]&=&\int_0^\infty \mathbb{E}\left[\mathbf{1}_{\{t<\tau\}}\mathbf{1}_{\{\theta\leq t\}}\right] \, \mathrm{d}t\\
&=&\int_0^\infty \mathbb{E}\left[\mathbf{1}_{\{t<\tau\}}\mathbb{E}\left[\mathbf{1}_{\{\theta\leq t\}}\, |\, \mathcal{F}_t\right]\right] \, \mathrm{d}t\\
&=&\mathbb{E}\left[\int_0^\tau \mathbb{P}\left(\theta\leq t \, | \, \mathcal{F}_t\right) \, \mathrm{d}t\right].\\
\end{eqnarray*}
Furthermore, for any $t\geq 0$,
\begin{eqnarray*}
\mathbb{P}\left(\theta\leq t \, | \, \mathcal{F}_t\right)&=&\mathbb{P}\left( \left. \sup_{s\geq t}X_s\leq \XmaxT \, \right| \, \mathcal{F}_t\right)\\
&=&\mathbb{P}\left(S+X_t\leq \XmaxT \, \mid \, \mathcal{F}_t\right)\\
&=&F(Y_t^0),
\end{eqnarray*}
where $S$ denotes an independent copy of $\Xmax$.
We conclude that when $\tau$ has finite mean
\[\mathbb{E}[|\theta-\tau|]=2\mathbb{E} \left[ \int_0^{\tau}  F(Y_t^0)\, \mathrm{d}t \right] + \mathbb{E}[\theta]-\mathbb{E}[\tau] = \mathbb{E} \left[ \int_0^{\tau} \left( 2 F(Y_t^0)-1 \right) \, \mathrm{d}t \right] + \mathbb{E}[\theta]. \]
\end{proof}

Hence, by defining a function $V$ on $\mathbb{R}_{\geq 0}$ as
\begin{equation}\label{def_V}
V(y) = \inf_{\tau} \mathbb{E} \left[ \int_0^{\tau} \left( 2 F(Y^y_t)-1 \right) \, \mathrm{d}t \right]
\end{equation}
we have that an optimal stopping time for $V(0)$ is also optimal in (\ref{time_sup_def}). Therefore, let us analyse the function $V$.

Inspecting the integrand in (\ref{def_V}) makes it clear that a quantity of interest is the (lower) median of the law of $\Xmax$, that is:

\[ \mathrm{m} = \inf\{ z\geq 0 \, | \,F(z) \geq 1/2 \}. \]
If $\mathrm{m}=0$, that is if $\mathbb{P}(\Xmax=0) \geq 1/2$, it is easy to see it is optimal to stop immediately. Indeed, we have the following result.
\begin{prop} The time $\tau=0$ is optimal in (\ref{def_V}) for all $y \geq 0$ if and only if $\mathrm{m}=0$.
 In this case $V(y)=0$ for all $y \geq 0$.
\end{prop}

\begin{proof} Suppose $\mathrm{m}=0$. This implies $F(z) \geq 1/2$ for all $z>0$, and, in particular, also $F(0) \geq 1/2$ (by right continuity). Hence $2F(z)-1 \geq 0$ for all $z \geq 0$ and the result follows.
Next suppose $\mathrm{m}>0$. We have

\[V(0) \leq \mathbb{E} \left[ \int_0^{\sigma(0,\mathrm{m})} \left( 2 F(Y_t^0)-1 \right) \, \mathrm{d}t \right]<0,\]
since $\sigma(0,\mathrm{m})>0$ a.s. by right continuous paths of $Y^0$ and $F<1/2$ on $[0,\mathrm{m})$.
\end{proof}

We now turn our attention to the more interesting case $\mathrm{m}>0$. Note that it is still possible that $F$ has a discontinuity in $0$ (but with size strictly less than $1/2$). Recall our standing assumptions that $X$ drifts to $-\infty$ and $\theta$ has finite mean. Recall also that Lemma \ref{lem_Fcont} states that $F$ is Lipschitz continuous on $\mathbb{R}_{\geq 0}$ at least when $X$ creeps upwards. In the result below we denote by $V'_-$ and $V'_+$ the left and right derivative of $V$, respectively.

\begin{theorem}\label{thm_main} Suppose that $X$ is not a compound Poisson process and is such that $\mathrm{m}>0$. Then there exists a $y^* \in [\mathrm{m},\infty)$ such that an optimal stopping time in (\ref{def_V}) is given by

\[ \tau^* = \inf \{ t \geq 0 \, | \, V(Y^y_t)=0 \} = \inf \{ t \geq 0 \, | \, Y^y_t \geq y^* \}. \]
Furthermore $V$ is a non-decreasing, continuous function satisfying the following:
\begin{itemize}
 \item[(i)] if $X$ is regular downwards and $F$ is Lipschitz continuous on $\mathbb{R}_{\geq 0}$, then $y^*>\mathrm{m}$ and $V_{-}'(y^*)=V_{+}'(y^*)=0$ (smooth pasting);
\item[(ii)] if $X$ is irregular downwards, then $y^*>\mathrm{m}$ is the unique solution on $\mathbb{R}_{>0}$ to the equation

\[ \mathbb{E} \left[ \int_0^{\sigma_+(y)} \left( 2F(Y^{y}_u)-1 \right) \, \mathrm{d}u \right]=0. \]
Furthemore, when $F'$ exists and is positive on $\mathbb{R}_{>0}$, smooth pasting does not hold, i.e. $V_{-}'(y^*)>V_{+}'(y^*)=0$.
\end{itemize}
\end{theorem}

\begin{remark}
As the median of the distribution plays an important role in our main result, it is interesting to compare this with the ``median rule" in \cite{Peskir10}. There, the so-called ``hidden target'' is a random variable which is taken to be independent of the underlying process $X$ (which is assumed to be continuous in \cite{Peskir10}) and the aim is to stop as close as possible (in the sense described in \cite{Peskir10}) to this hidden target. It turns out there that it is optimal to stop as soon as the $X$ hits the median of the hidden target. In our setting, the target clearly depends on the whole path of $X$ and here it turns out that $y^*>\mathrm{m}$, i.e. we should wait a bit longer than just the first hitting time of the median (at least with with positive probability).
\end{remark}

\begin{remark}\label{rem_Kevin} In \cite{DuToit08} it was postulated that the same problem we consider here but with a finite rather than an infinite time horizon is solved by the first time the process $Y^y$ exceeds a time dependent boundary. As the time dependency of the boundary is a consequence of the finite horizon, this would in the case of an infinite time horizon naturally extend to the first time the process $Y^y$ exceeds a (time independent) level. Indeed, the above Theorem \ref{thm_main} confirms this observation.
\end{remark}

\section{Proof of Theorem \ref{thm_main}}\label{sec_proof}
\setcounter{equation}{0}

As the proof of Theorem \ref{thm_main} is rather long, we break it into a number of lemmas which we prove in this section. We still have the standing assumption (\ref{henceforthassumption}). Throughout this section we denote the payoff process for any $y \geq 0$ by

\[ L^y_t = \int_0^{t} \left( 2 F(Y^y_u)-1 \right) \, \mathrm{d}u \quad \text{for all $t \geq 0$} \]
so that $V(y)=\inf_\tau \mathbb{E}[L^y_\tau]$. Note that for any $y \geq 0$, since $Y^y$ drifts to $\infty$, $L^y$ also drifts to $\infty$. Furthermore, recall the notation introduced in (\ref{def_stopTimes}), and for $x>0$ we write $\rho(x)$ for the last time $Y^0$ is in the interval $[0,x]$, i.e.

\[ \rho(x)=\sup\{t \geq 0 \, | \, Y_t^0 \leq x\}. \]

We start with a technical result.

\begin{lemma}\label{lem_lastExit} For any $x > 0$ we have $\mathbb{E}[\rho(x)]<\infty$.
\end{lemma}

\begin{proof} Note that $\rho(x)=\theta + \zeta(x)$, where $\zeta(x)$ is the time the final excursion of $Y^0$ leaves $[0,x]$. As we assume that $\theta$ has a finite mean, the post-maximum process has the same law as $-X$ conditioned to stay positive (see for example Proposition 4.4 in \cite{Millar77} or Theorem 3.1 in \cite{Bertoin93}). Therefore, $\zeta(x)$ is the last passage time over the level $x$ of $-X$ conditioned to stay positive. From p.357 in \cite{Duquesne03} and Lemma 4 in \cite{Bertoin91} we know that $\zeta(x)$ is equal in distribution to $\overline{g}_{\hat{\tau}^+(x)}$, where $\hat{\tau}^+(x)$ denotes the first passage time of $-X$ over level $x$ and $\overline{g}_t$ denotes the time of the last maximum of $-X$ prior to time $t>0$. Therefore

\[ \mathbb{E}[\zeta(x)]=\mathbb{E}[\overline{g}_{\hat{\tau}^+(x)}]\leq\mathbb{E}[\hat{\tau}^+(x)]<\infty, \]
where the last inequality holds since $-X$ drifts to $+\infty$ (see for example \cite{Bertoin96} Proposition 17 on p. 172).
\end{proof}

Next, we derive some properties of $V$.

\begin{lemma}\label{lem_prePropV} The function $V$ is non-decreasing with $V(y) \in (-\infty,0]$ for all $y \geq 0$. In particular, $V(y)<0$ for any $y \in [0,\mathrm{m})$.
\end{lemma}

\begin{proof} From the monotonicity of $Y^y$ in $y$ and $z \mapsto 2F(z)-1$ it is clear that $V$ is non-decreasing. For any $y \geq 0$ we have $V(y) \leq \mathbb{E}[L^y_0]=0$. In particular, take some $y \in [0,\mathrm{m})$ and let $y_0 \in (y,\mathrm{m})$. Then $2F(z)-1 \leq 2F(y_0)-1<0$ for all $z \in [0,y_0]$ and hence $V(y) \leq \mathbb{E}[L^y_{\sigma(y,y_0)}] \leq (2F(y_0)-1) \mathbb{E}[\sigma(y,y_0)]<0$, where the last inequality uses that by right continuity of paths we have $\sigma(y,y_0)>0$ a.s.

Furthermore,

\[ V(0) \geq -\mathbb{E} \left[ \int_0^{\infty} \mathbf{1}_{\{ Y^0_u \leq \mathrm{m} \}} \, \mathrm{d}u \right]
\geq -\mathbb{E}[\rho(\mathrm{m})] \]
and using Lemma \ref{lem_lastExit} it follows that $V(0)>-\infty$. The monotonicity of $V$ implies that $V(y)>-\infty$ for all $y \geq 0$.
\end{proof}

The following result is helpful to prove continuity of $V$.

\begin{lemma}\label{lem_yBar} There exists a $\bar{y}$ (sufficiently large) such that for all $y \geq 0$

\[ V(y) = \inf_{\tau} \mathbb{E} \left[ L^y_{\tau \wedge \sigma(y,\bar{y})} \right]. \]
\end{lemma}

\begin{proof} We first show that there exists a $\bar{y}$ such that

\begin{equation}\label{K3mei2}
V(y) = 0 \quad \text{for all $y \geq \bar{y}$.}
\end{equation}
For this, denote $y_0=F^{-1}(3/4)$ and

\[ \sigma_0 = \inf \{ t \geq 0 \, | \, Y^y_t \leq y_0 \} = \inf \{ t \geq 0 \, | \, X_t \geq y-y_0 \}. \]
Note that

\[ L^y_t = \int_0^t \left( 2F(Y^y_u)-1 \right) \, \mathrm{d}u \geq \frac{1}{2}t \quad \text{for $t \leq \sigma_0$.} \]
Using this together with a dynamic programming argument and that $V(y) \geq V(0)$ for all $y \geq 0$ (cf. Lemma \ref{lem_prePropV}) it follows (note that in below computation we consider w.l.o.g. only stopping times $\tau$ which are finite a.s.)

\begin{eqnarray*}
V(y) &=& \inf_{\tau} \mathbb{E} \left[ \mathbf{1}_{\{ \tau < \sigma_0 \}} L^y_\tau + \mathbf{1}_{\{ \tau \geq \sigma_0 \}} L^y_{\sigma_0} + \mathbf{1}_{\{ \tau \geq \sigma_0 \}} \left( L^y_\tau - L^y_{\sigma_0} \right) \right] \nonumber\\
&=& \inf_{\tau} \mathbb{E} \left[ L^y_{\tau \wedge \sigma_0} + \mathbf{1}_{\{ \tau \geq \sigma_0 \}} V \left( Y^y_{\sigma_0} \right) \right] \nonumber\\
&\geq & \inf_{\tau} \mathbb{E} \left[ \frac{1}{2} (\tau \wedge \sigma_0) + \mathbf{1}_{\{ \tau \geq \sigma_0 \}} V(0) \right] \nonumber\\
&=& \inf_{\tau} \mathbb{E} \left[ \frac{1}{2} \tau + \mathbf{1}_{\{ \tau \geq \sigma_0 \}} V(0) \right]. \nonumber
\end{eqnarray*}
Appealing to Lemma \ref{lemmm} with $a=1/2$ and $b=V(0)<0$, the ultimate right-hand side is equal to $f(y_0-y)$ which indeed vanishes for all $y \geq y_0-x_0$ (with $f$ and $x_0$ as defined in Lemma \ref{lemmm}). Hence, if we set $\bar{y}=y_0-x_0$ we indeed arrive at (\ref{K3mei2}).

Using (\ref{K3mei2}) together with another dynamic programming argument and $Y^y_{\sigma(y,\bar{y})} \geq \bar{y}$ we can now derive the result:

\begin{eqnarray*}
V(y) &=& \inf_{\tau} \mathbb{E} \left[ \mathbf{1}_{\{ \tau < \sigma(y,\bar{y}) \}} L^y_\tau + \mathbf{1}_{\{ \tau \geq \sigma(y,\bar{y}) \}} L^y_{\sigma(y,\bar{y})} + \mathbf{1}_{\{ \tau \geq \sigma(y,\bar{y}) \}} \left( L^y_\tau - L^y_{\sigma(y,\bar{y})} \right) \right] \nonumber\\
&=& \inf_{\tau} \mathbb{E} \left[ \mathbf{1}_{\{ \tau < \sigma(y,\bar{y}) \}} L^y_\tau + \mathbf{1}_{\{ \tau \geq \sigma(y,\bar{y}) \}} L^y_{\sigma(y,\bar{y})} + \mathbf{1}_{\{ \tau \geq \sigma(y,\bar{y}) \}} V\left( Y^y_{\sigma(y,\bar{y})} \right) \right] \nonumber\\
&=& \inf_{\tau} \mathbb{E} \left[ \mathbf{1}_{\{ \tau < \sigma(y,\bar{y}) \}} L^y_\tau + \mathbf{1}_{\{ \tau \geq \sigma(y,\bar{y}) \}} L^y_{\sigma(y,\bar{y})} \right] \nonumber\\
&=& \inf_{\tau} \mathbb{E} \left[ L^y_{\tau \wedge \sigma(y,\bar{y})} \right]. \nonumber
\end{eqnarray*}
\end{proof}

Next, we use the result above to show the continuity of $V$.

\begin{lemma}\label{lem_contV} The function $V$ is continuous.
\end{lemma}

\begin{proof} From the above Lemma \ref{lem_yBar} we know we may write

\begin{equation}\label{3aug1}
V(y) = \inf_{\tau} \mathbb{E} \left[ L^y_{\tau \wedge \sigma(y,\bar{y})} \right].
\end{equation}
Since $L^y$ is a continuous process and

\begin{equation}\label{24mei1}
\mathbb{E} \left[ \sup_{t \geq 0} \left| L^y_{t \wedge \sigma(y,\bar{y})} \right| \right] \leq  \mathbb{E}[\sigma(y,\bar{y})]<\infty
\end{equation}
(note that the last inequality follows for instance from Lemma \ref{lem_lastExit}) it is clear that the infimum in (\ref{3aug1}) is attained. As $F$ is continuous on $\mathbb{R}_{\geq 0}$ (cf. Lemma \ref{lem_Fcont}) it is uniformly continuous on $[0,\bar{y}]$. Take any $\eps>0$. Let $\delta>0$ be such that for all $y_1,y_2 \in [0,\bar{y}]$ with $|y_1-y_2|<\delta$ it holds $|F(y_1)-F(y_2)|<\eps$. For any $y \geq 0$ we have, where $\tau_y$ is the optimal stopping time when starting from $y$ and we use $Y^{y+\delta}_t-Y^{y}_t \leq \delta$ for all $t \geq 0$:

\[
V(y+\delta)-V(y) \leq \mathbb{E}[L^{y+\delta}_{\tau_y}] - \mathbb{E}[L^{y}_{\tau_y}] \leq 2 \mathbb{E} \left[ \int_0^{\sigma(y,\bar{y})} \left( F(Y^{y+\delta}_u)-F(Y^{y}_u) \right) \, \mathrm{d}u \right] \leq 2 \eps \mathbb{E}[\sigma(y,\bar{y})],
\]
establishing the continuity of $V$ as $\mathbb{E}[\sigma(y,\bar{y})]<\infty$.
\end{proof}

The continuity of $V$ allows us to show that an optimal stopping time for  is given by the first time the reflected process exceeds a certain threshold.

\begin{lemma}\label{ost} Denoting $y^*= \inf \{ y \geq 0 \, | \, V(y)=0 \} \in [\mathrm{m},\bar{y}]$, we have that for any $y \geq 0$ the stopping time

\[ \sigma(y,y^*)=\inf \{ t \geq 0 \, | \, Y^y_t \geq y^* \}\]
attains the infimum in $V(y)=\inf_{\tau} \mathbb{E}[L^y_\tau]$.
\end{lemma}

\begin{proof} Following the usual arguments from general theory for optimal stopping (see e.g. Theorem 2.4 on p. 37 in \cite{Peskir06} for details), taking into account (\ref{24mei1}) and the fact that $L^y$ is continuous, an optimal stopping time for $V(y)$ is given by $\tau^*=\inf\{ t \geq 0 \, | \, \widehat{L}^y_t=L^y_t \}$. Here $\widehat{L}^y$ denotes the the Snell envelope of $L^y$ which satisfies on account of the Markov property of $(L^y,Y^y)$:

\[ \widehat{L}^y_t = \essinf{\tau \geq t} \mathbb{E} \left[ L^y_{\tau} \, | \, \mathcal{F}_t \right] = L^y_t + \essinf{\tau \geq t} \mathbb{E} \left[ L_{\tau}^y-L^y_t \, | \, \mathcal{F}_t \right] = L^y_t+V(Y^y_t). \]
Hence, we may also write $\tau^* = \inf \{ t \geq 0 \, | \, V(Y^y_t)=0 \}$. From the Lemmas \ref{lem_prePropV}, \ref{lem_yBar} and \ref{lem_contV} above it follows we may indeed write $\tau^*=\sigma(y,y^*)$ where $y^*$ is as defined above and $y^* \in [\mathrm{m},\bar{y}]$.
\end{proof}

The last lemma in this section concludes the proof. It concerns the smoothness of the function $V$ at $y^*$ and an expression for $y^*$ if $X$ is irregular downwards.
\begin{lemma} We have the following.
\begin{enumerate}
\item
If $X$ is regular downwards and $F$ is Lipschitz continuous on $\mathbb{R}_{\geq 0}$, then there is smooth pasting at $y^*$ i.e. $V^\prime(y^*-)=0$.
\item
If $X$ is irregular downwards, then $y^*$ is the unique solution on $\mathbb{R}_{>0}$ to the equation

\[ \mathbb{E} \left[ \int_0^{\sigma_+(y)} \left( 2F(Y^{y}_u)-1 \right) \, \mathrm{d}u \right]=0. \]
Furthemore, if $F'$ exists and is positive on $\mathbb{R}_{>0}$ smooth pasting does not hold, i.e. $V_{-}'(y^*)>0$.
 \end{enumerate}
In both cases above we have $y^*>\mathrm{m}$.
\end{lemma}

\begin{proof}
Consider case (i). As $V$ is non-decreasing, for $V_{-}'(y^*)=0$ it suffices to show that

\begin{equation}\label{9aug2}
\limsup_{\eps \downarrow 0} \frac{V(y^*)-V(y^*-\eps)}{\eps} \leq 0.
\end{equation}
For any $\eps>0$ we know from Lemma \ref{ost}

\[ V(y^*-\eps)=\mathbb{E}[L^{y^*-\eps}_{\sigma(y^*-\eps,y^*)}] \quad \text{and} \quad V(y^*) \leq \mathbb{E}[L^{y^*}_{\sigma(y^*-\eps,y^*)}] \]
and hence, if $C$ is a Lipschitz constant of $F$

\begin{multline*}
V(y^*)-V(y^*-\eps) \leq 2 \mathbb{E} \left[ \int_0^{\sigma(y^*-\eps,y^*)} \left( F(Y^{y^*}_u)-F(Y^{y^*-\eps}_u) \right) \, \mathrm{d}u \right] \\
\leq 2 \mathbb{E} \left[ \int_0^{\sigma(y^*-\eps,y^*)} C \left( Y^{y^*}_u-Y^{y^*-\eps}_u \right) \, \mathrm{d}u \right] \leq 2C\eps \mathbb{E}[\sigma(y^*-\eps,y^*)].
\end{multline*}
Now (\ref{9aug2}) follows by noting that
\[ \sigma(y^*-\eps,y^*) \leq \inf \{ t \geq 0 \, | \, y^*-\eps-X_t \geq y^* \} \leq \inf \{ t \geq 0 \, | \, X_t \leq -\eps \} \downarrow 0 \text{ a.s. as $\eps \downarrow 0$} \]
on account of $X$ being regular downwards.

We now show that $y^*>\mathrm{m}$. For this, suppose we had $y^*=\mathrm{m}$. We will show that this violates the smooth pasting we have just established. For all $\eps>0$ small enough we have

\begin{equation}\label{Kevin5}
Y^{\mathrm{m}-\eps}_t \leq \mathrm{m} \text{ for } t<\tau^-(-\eps) \wedge \tau^+(\mathrm{m}-\eps).
\end{equation}
As $y^* \geq \mathrm{m}$ this ensures that $\tau^-(-\eps) \wedge \tau^+(\mathrm{m}-\eps) \leq \sigma(\mathrm{m}-\eps,y^*)$ and hence a dynamic programming argument yields

\[ V(\mathrm{m}-\eps) = \mathbb{E} \left[ \int_0^{\tau^-(-\eps) \wedge \tau^+(\mathrm{m}-\eps)}  \left( 2F(Y^{\mathrm{m}-\eps}_u)-1 \right)  \, \mathrm{d}u + V \left( Y^{\mathrm{m}-\eps}_{\tau^-(-\eps) \wedge \tau^+(\mathrm{m}-\eps)} \right) \right]. \]
The first integral in the above expectation is non-positive due to (\ref{Kevin5}) and hence

\begin{multline*}
V(\mathrm{m}-\eps) \leq \mathbb{E} \left[ V \left( Y^{\mathrm{m}-\eps}_{\tau^-(-\eps) \wedge \tau^+(\mathrm{m}-\eps)} \right) \right] \leq \mathbb{E} \left[ \mathbf{1}_{\{ \tau^+(\mathrm{m}-\eps) < \tau^-(-\eps)\}} V \left( Y^{\mathrm{m}-\eps}_{\tau^-(-\eps) \wedge \tau^+(\mathrm{m}-\eps)} \right) \right] \\
= V(0) \mathbb{P}(\tau^+(\mathrm{m}-\eps)<\tau^-(-\eps)).
\end{multline*}
Appealing to Lemma \ref{limsuph} and using that $V(0)<0$ we see that this implies $V_{-}'(y^*)=V_{-}'(\mathrm{m})>0$, which is the required contradiction.

Now consider case (ii). Note that we have $\sigma_+(y)>0$ a.s. for all $y>0$ (as $X$ is irregular downwards and $\overline{X}_t \downarrow 0$ as $t \downarrow 0$). For any $y<y^*$ we have $\sigma_+(y) \leq \sigma(y,y^*)$ and hence by a dynamic programming argument we may write

\[ V(y) = \mathbb{E} \left[ L^y_{\sigma_+(y)} + V(Y^y_{\sigma_+(y)}) \right]. \]
Letting $y \uparrow y^*$ and using that $V(y)=0$ for all $y \geq y^*$ we see (recall $V$ is continuous)

\begin{equation}\label{7aug1}
V(y^*-) = \mathbb{E} \left[ \int_0^{\sigma_+(y^*)} \left( 2F(Y^{y^*}_u)-1 \right) \, \mathrm{d}u \right].
\end{equation}
By continuity of $V$ we have $V(y^*-)=0$ and hence the above right-hand side vanishes. Furthermore, since the integrand is monotone in $y$ and for $y_1<y_2$ we have $\sigma_+(y_2) \geq \sigma_+(y_1)$, the inequality being strict with positive probability, it is clear that $y^*$ is the unique element in $\mathbb{R}_{>0}$ for which the right-hand side of (\ref{7aug1}) vanishes.

It is clear that $y^*>\mathrm{m}$. Indeed if we had $y^*=\mathrm{m}$ then by continuity of $V$ this would imply $V(\mathrm{m}-)=0$. However the right hand side of (\ref{7aug1}) is strictly negative for $y^*=\mathrm{m}$ on account of the following facts: it holds that $F(z)<1/2$ for $z<\mathrm{m}$, $Y^{\mathrm{m}}_t<\mathrm{m}$ for $0<t<\sigma_+(\mathrm{m})$ and $\sigma_+(\mathrm{m})>0$ a.s.

Finally, let us show that smooth pasting does not hold when $F$ has a positive derivative on $\mathbb{R}_{>0}$. Indeed, we have for any $\eps>0$

\[ V(y^*-\eps) \leq \mathbb{E} \left[ \int_0^{\sigma_+(y^*)}  \left( 2F(Y^{y^*-\eps}_u)-1 \right)  \, \mathrm{d}u \right] \]
and using (\ref{7aug1}) we get

\[ V(y^*)-V(y^*-\eps) \geq \mathbb{E} \left[ \int_0^{\sigma_+(y^*)}  2 \left( F(Y^{y^*}_u)-F(Y^{y^*-\eps}_u) \right) \, \mathrm{d}u \right]. \]
Dividing by $\eps$ and applying Fatou's Lemma yields

\begin{equation}\liminf_{\eps \downarrow 0} \frac{V(y^*)-V(y^*-\eps)}{\eps} \geq 2 \mathbb{E} \left[ \int_0^{\sigma_+(y^*)} \liminf_{\eps \downarrow 0} \mathbf{1}_{\{ \overline{X}_u < y^*-\eps \}} \frac{F(Y^{y^*}_u)-F(Y^{y^*}_u-\eps)}{\eps} \, \mathrm{d}u \right]. \label{ietsrhs}\end{equation}
As $Y^{y^*}_u-Y^{y^*-\eps}_u=\eps$ on the event $\{ \overline{X}_u < y^*-\eps \}$ we see that the right-hand side in (\ref{ietsrhs}) is indeed strictly positive since $F'$ is.
\end{proof}

\section{The case that $\theta$ has infinite mean}\label{sec_infmean}
\setcounter{equation}{0}

In this section we show that in the case when $\theta$ has infinite mean, it is impossible to find a stopping time which has finite $L^1$-distance to $\theta$. This is intuitively not very surprising given Theorem \ref{thm_main}. Indeed, suppose we approximate $X$ in a suitable sense by a sequence of L\'evy processes, indexed by $n$ say, for each of which the corresponding time of the ultimate supremum $\theta_n$ has finite mean. For each element in the sequence, a stopping time minimising the $L^{1}$-distance to $\theta_n$ is the first time the reflected process exceeds a level $y^*_n$. Suppose the $y^*_n$'s have a limit $y^*_{\infty}$. If $y^*_{\infty}$ is finite then the limit of the optimal stopping times, say $\hat{\tau}$, is the first time the reflected process associated with $X$ exceeds the level $y^*_{\infty}$. However this would mean that $\hat{\tau}$ has finite mean and hence $\mathbb{E}[|\theta-\hat{\tau}|]=\infty$. On the other hand, if $y^*_{\infty}$ is infinite then $\hat{\tau}=\infty$ a.s. and hence still $\mathbb{E}[|\theta-\hat{\tau}|]=\infty$.

We will prove the following result by introducing an independent, exponentially time horizon.

\begin{prop}\label{prop_theta_inf} Suppose as before that $X$ is not a compound Poisson process and drifts to $-\infty$. Suppose now that $\mathbb{E}[\theta]=\infty$. Then (\ref{time_sup_def}) is degenerate, i.e. for all stopping times $\tau$ it holds $\mathbb{E}[|\theta-\tau|]=\infty$.
\end{prop}

\begin{proof} For any $q>0$, let $\mathrm{e}(q)$ denote an exponentially distributed random variable with mean $1/q$, independent of $X$. (For convenience we denote the joint law of $X$ and $\mathrm{e}(q)$ also by $\mathbb{P}$). Furthermore, for any random time $T$ we denote

\[ T^{(q)} = T \wedge \mathrm{e}(q). \]

Let us assume that a stopping time $\hat{\tau}$ exists with $\mathbb{E}[|\theta-\hat{\tau}|]<\infty$ and derive a contradiction. First, note that since $|\theta^{(q)}-\hat{\tau}^{(q)}| \to |\theta-\hat{\tau}|$ a.s. as $q \downarrow 0$ and $|\theta^{(q)}-\hat{\tau}^{(q)}| \leq |\theta-\hat{\tau}|$ for all $q>0$ (this is readily checked from the definition of $\theta^{(q)}$ and $\hat{\tau}^{(q)}$) dominated convergence yields

\begin{equation}\label{11jun3}
\limsup_{q \downarrow 0} \inf_{\tau} \mathbb{E}[|\theta^{(q)}-\tau^{(q)}|] \leq \lim_{q \downarrow 0} \mathbb{E} \left[ |\theta^{(q)}-\hat{\tau}^{(q)}| \right] = \mathbb{E}[|\theta-\hat{\tau}|]<\infty.
\end{equation}
Now, using that for any $t \geq 0$ we have

\[ \mathbb{P}(\theta^{(q)} \leq t \, | \, \mathcal{F}_t) = \mathbb{P}(\mathrm{e}(q) \leq t) + \mathbb{P}(\mathrm{e}(q)>t) \mathbb{P}(\theta \leq t \, | \, \mathcal{F}_t) =1-e^{-qt} + e^{-qt} F(Y^0_t), \]
the same reasoning as in Proposition \ref{prop_repr} yields for any stopping time $\tau$

\begin{multline}\label{11jun4}
\mathbb{E}[|\theta^{(q)}-\tau^{(q)}|] = \mathbb{E}[\theta^{(q)}] + \mathbb{E} \left[ \int_0^{\tau^{(q)}} \left( 1+2e^{-qu} \left( F(Y^0_u)-1 \right) \right) \, \mathrm{d}u \right] \\
= \mathbb{E}[\theta^{(q)}] + \mathbb{E} \left[ \int_0^{\tau} e^{-qu} \left( 1+2e^{-qu} \left( F(Y^0_u)-1 \right) \right) \, \mathrm{d}u \right].
\end{multline}
To examine the right-hand side of (\ref{11jun4}), define the function $V_q$ on $\mathbb{R}_{\geq 0} \times \mathbb{R}_{\geq 0}$ as

\begin{equation}\label{11jun2}
V_q(t,y) = \inf_{\tau} \mathbb{E} \left[ \int_0^{\tau} e^{-qu} \left( 1+2e^{-q(t+u)} \left( F(Y^y_u)-1 \right) \right) \, \mathrm{d}u \right].
\end{equation}
Note that the mappings $t \mapsto V_q(t,y)$ for any fixed $y \geq 0$ and $y \mapsto V_q(t,y)$ for any fixed $t \geq 0$ are non-decreasing. Furthermore for any $t \geq 0$

\begin{equation}\label{11jun1}
V_q(t,0) \geq -\int_0^{\infty} e^{-qu} \, \mathrm{d}u > -\infty
\end{equation}
and hence $V_q$ is a bounded function taking values in $\mathbb{R}_{\leq 0}$. It is straightforward to adjust slightly the arguments from the proof of Lemma \ref{lem_contV} (using (\ref{11jun1}))  to see that $V_q$ is a continuous function. Following the same arguments as in the proof of Lemma \ref{ost}, the Snell envelope $\widehat{L}$ of the process $L$ defined as

\[ L_t = \int_0^{t} e^{-qu} \left( 1+2e^{-qu} \left( F(Y^0_u)-1 \right) \right) \, \mathrm{d}u \quad \text{for all $t \geq 0$}\]
satisfies

\begin{multline*}
\widehat{L}_t = \essinf{\tau \geq t} \mathbb{E}[L_{\tau} \, | \, \mathcal{F}_t ] = L_t + \essinf{\tau \geq t} \mathbb{E}[L_{\tau}-L_t \, | \, \mathcal{F}_t ] \\
= L_t + e^{-qt} \, \essinf{\tau \geq t} \mathbb{E}\left[ \left. \int_t^{\tau} e^{-q(u-t)} \left( 1+2e^{-qu} \left( F(Y^0_u)-1 \right) \right) \, \mathrm{d}u  \, \right| \, \mathcal{F}_t \right] \\
= L_t + e^{-qt} V_q(t,Y^0_t)
\end{multline*}
for any $t \geq 0$. Therefore general theory of optimal stopping dictates that a stopping time minimising the right-hand side of (\ref{11jun4}) and, equivalently, which is optimal for $V_q(0,0)$, is given by

\begin{equation}\label{11jun6}
\tau^*_q = \inf \{ t \geq 0 \, | \, \widehat{L}_t=L_t \} = \inf \{ t \geq 0 \, | \, V_q(t,Y^0_t)=0 \} = \inf \{ t \geq 0 \, | \, Y^0_t \geq b_q(t) \},
\end{equation}
where $b_q(u)=\inf \{ y \geq 0 \, | \, V_q(u,y)=0 \} >0$ for all $u \geq 0$. Note that the final step uses the monotonicity of $V_q$ in $y$. (Cf. Theorem 2.4 on p. 37 in \cite{Peskir06} for details.) Note that the monotonicity of $t \mapsto V_q(t,y)$ implies that $b_q$ is non-increasing.

Now we are ready to return to our main argument. Taking the infimum over $\tau$ in (\ref{11jun4}) yields

\[ \inf_{\tau} \mathbb{E}[|\theta^{(q)}-\tau^{(q)}|] = \mathbb{E}[\theta^{(q)}] + V_q(0,0), \]
where the infimum in the left-hand side (and in $V_q(0,0)$) is attained by $\tau^*_q$ as defined in (\ref{11jun6}). Letting $q \downarrow 0$ in this equation it follows on account of (\ref{11jun3}) and $\theta^{(q)} \uparrow \theta$ so $\mathbb{E}[\theta^{(q)}] \to \mathbb{E}[\theta]=\infty$ that $V_q(0,0) \to -\infty$ as $q \downarrow 0$.

Next, we show that this implies

\begin{equation}\label{11jun7}
\text{for any $u_0>0$ we have $b_q(u_0) \to \infty$ as $q \downarrow 0$.}
\end{equation}
Indeed, suppose this were not the case, i.e. that $(q_k)_{k \geq 0}$ exists with $q_k \downarrow 0$ as $k \to \infty$  such that $b_{q_k}(u_0) \leq a$ for some $a>0$ and all $k \geq 0$. The monotonicity of $b$ then implies $b_{q_k}(u) \leq a$ for all $u \geq u_0$ and $k \geq 0$. Hence for all $k \geq 0$

\[ \tau^*_{q_k} \leq \inf \{ u \geq u_0 \, | \, Y^0_u \geq a \}, \]
where the stopping time in the right-hand side has finite mean due to the fact that $X$ drifts to $-\infty$ (see Lemma \ref{lem_lastExit}). However, due to (\ref{11jun2}) it holds that $V_{q_k}(0,0) \geq -\mathbb{E}[\tau^*_{q_k}]$, violating $V_{q_k}(0,0) \to -\infty$ as $q_k \downarrow 0$. Hence (\ref{11jun7}) holds.

For any $u_0>0$ we see from (\ref{11jun7}) together with the monotonicity of $b_q$ that

\[ \inf_{u \in [0,u_0]} b_q(u) \geq b_q(u_0) \to \infty \quad \text{ as $q \downarrow 0$} \]
and consequently

\[ \mathbb{P}(\tau^*_q \leq u_0) \leq \mathbb{P} \left( \sup_{s \leq u_0} Y^0_s \geq b_q(u_0) \right) \to 0 \quad \text{ as $q \downarrow 0$}. \]
This implies that

\begin{equation}\label{22jun1}
\inf_{\tau} \mathbb{E}[|\theta^{(q)}-\tau^{(q)}|] = \mathbb{E}[|\theta^{(q)}-\tau^{* (q)}_q|] \to \infty \quad \text{ as $q \downarrow 0$.}
\end{equation}
Indeed, one way to see that (\ref{22jun1}) holds is the following. Fix some $x>0$. For an arbitrary $\eps>0$ pick $\theta_0$ so that $\mathbb{P}(\theta>\theta_0) \leq \eps$. Then

\begin{eqnarray*}
\mathbb{P} ( |\theta^{(q)}-\tau^{* (q)}_q| \leq x ) & \leq & \mathbb{P}(\theta>\mathrm{e}(q)) + \mathbb{P}(|\theta-\tau^{* (q)}_q| \leq x) \\
& \leq & \mathbb{P}(\theta>\mathrm{e}(q)) + \mathbb{P}(\theta>\theta_0) + \mathbb{P}(\tau^{* (q)}_q \leq x+\theta_0) \\
& \leq & \mathbb{P}(\theta>\mathrm{e}(q)) + \mathbb{P}(\theta>\theta_0) + \mathbb{P}(\mathrm{e}(q) \leq x+\theta_0) + \mathbb{P}(\tau^{*}_q \leq x+\theta_0),
\end{eqnarray*}
and as $q \downarrow 0$ all the terms in the final right-hand side vanish except for the second, which is bounded above by the arbitrarily chosen $\eps$.

As (\ref{22jun1}) violates (\ref{11jun3}) we have arrived at the required contradiction.
\end{proof}

\begin{remark} If $\theta$ has infinite mean a possibility is to replace the $L^1$-distance by a more interesting function. An alternative would for instance be to consider
\[\inf_\tau\mathbb{E}[|\tau-\theta|-\theta]\]
as done in \cite{Glover11}.
\end{remark}

\section{Spectrally negative L\'evy processes}\label{subsec_sn}
\setcounter{equation}{0}

One special case for which the results from Theorem \ref{thm_main} can be expressed more explicitly is when $X$ is spectrally negative, i.e. when the L\'evy measure $\Pi$ is concentrated on $\mathbb{R}_{<0}$ but $X$ is not the negative of a subordinator. In this section $X$ is assumed to be spectrally negative. Further details of the definitions and properties used in this section can be found in \cite{Kyprianou06} Chapter 8. It is worth recalling at this point that the relationship between the problem we set out to solve, i.e. (\ref{def_main1}), and the function $V$ which was defined in (\ref{def_V}) and analysed in Theorem \ref{thm_main} is

\[ \inf_{\tau} \mathbb{E}[|\theta-\tau|] = V(0) + \mathbb{E}[\theta]. \]
In particular, the stopping time that is optimal for $V(0)$ (cf. Theorem \ref{thm_main}) is also optimal for our original problem (\ref{def_main1}).

Let $\psi$ be the Laplace exponent of $X$, i.e.

\[ \psi(z)= \frac{1}{t} \log \mathbb{E}\left[ e^{zX_t} \right]. \]
Then $\psi$ exists at least on $\mathbb{R}_{\geq 0}$, it is strictly convex and infinitely differentiable with $\psi(0)=0$ and $\psi(\infty)=\infty$. Denoting by $\Phi$ the right inverse of $\psi$, i.e. for all $q \geq 0$

\[ \Phi(q)=\sup \{ z \geq 0 \, | \, \psi(z)=q \}, \]
the ultimate supremum $\Xmax$ follows an exponential distribution with parameter $\Phi(0)$ (with the usual convention that $\Xmax=\infty$ a.s. when $\Phi(0)=0$). It follows that

\begin{equation}\label{25mei1}
\Xmax<\infty \quad \Leftrightarrow \quad \Phi(0)>0 \quad \Leftrightarrow \quad \psi'(0+)<0.
\end{equation}
If the properties in (\ref{25mei1}) hold then the assumptions in Theorem \ref{thm_main} are satisfied. Indeed, $\Xmax<\infty$ implies that $X$ drifts to $-\infty$ (recall (\ref{18mei1})), and any compound Poisson process with no positive jumps is the negative of a subordinator which is excluded from the definition of spectrally negative. Furthermore, as $\Xmax$ follows an exponential distribution with parameter $\Phi(0),$ we have $\mathrm{m}>0$. Finally, Corollary 8.9 in \cite{Kyprianou06} yields

\[ \int_0^{\infty} \mathbb{P}(X_t \geq 0) \, \mathrm{d}t = \lim_{q \downarrow 0} \int_0^{\infty} e^{-qt} \mathbb{P}(X_t \geq 0) \, \mathrm{d}t = \lim_{q \downarrow 0} \int_0^{\infty} \Phi'(q) e^{-\Phi(q)x} \, \mathrm{d}x \]
which is finite when $\Phi(0)>0$, implying that $\theta$ has finite mean (recall (\ref{21jun1})).

Next, we briefly introduce scale functions. The scale function $W$ associated with $X$ is defined as follows: it satisfies $W(x)=0$ for $x<0$ while on $\mathbb{R}_{\geq 0}$ it is continuous, strictly increasing and characterised by its Laplace transform:

\[ \int_0^{\infty} e^{-\beta x} W(x) \, \mathrm{d}x = \frac{1}{\psi(\beta)} \quad \text{for $\beta>\Phi(0)$.} \]
Furthermore on $\mathbb{R}_{> 0}$ the left and right derivatives of $W$ exist. Note that in this case $X$ is regular (resp. irregular) downwards when $X$ is of unbounded (resp. bounded) variation. For ease of notation we shall assume that $\Pi$ has no atoms when $X$ is of bounded variation, which guarantees that $W\in C^1(0,\infty)$.
Also, when $X$ is of unbounded variation it holds that $W(0)=0$ with $W'(0+)\in(0,\infty]$ (see \cite{Lambert00}), otherwise $W(0)=1/\mathtt{d}$ where $\mathtt{d}>0$ is the drift of $X$.

For several families of spectrally negative L\'evy processes $W$ allows a (semi-)explicit representation, see \cite{Hubalek10} and the references therein. Scale functions are a natural tool for describing several types of fluctuation identities. Relevant for this paper is that the potential measure of the reflected process $Y^y$ starting from $y \geq 0$ killed at leaving the interval $[0,a]$, i.e.

\[ U_a(y,\mathrm{d}x) = \int_0^\infty \mathbb{P}(Y_t^y\in \mathrm{d}x, t<\sigma(y,a)) \, \mathrm{d}t,\]
can also be expressed in terms of scale functions (cf.  \cite{Kyprianou06} Theorem 8.11):

\begin{lemma}\label{lem_u}
When $X$ is spectrally negative  the measure $U_a(y,\mathrm{d}x)$ has a density on $(0,a)$ a version of which is given by
\[u_a(y,x)=W(a-y)\frac{W'(x)}{W'(a)}-W(x-y)\]
and only when $X$ is of bounded variation it has an atom at zero which is then given by
\[U_a(y,\{0\})=\frac{W(a-y)W(0)}{W'(a)}.\]
\end{lemma}

The results of Theorem \ref{thm_main} are expressed in terms of scale function as follows.

\begin{cor}\label{cor}
When $X$ is spectrally negative and satisfies any of the properties in (\ref{25mei1}), then $y^*$ is the unique solution on $\mathbb{R}_{>0}$ to the equation in $y$:
\begin{equation}\int_{[0,y]}(1-2e^{-\Phi(0)x})W(\mathrm{d}x)=0 \label{specnegy}\end{equation}
and
\begin{equation}
V(y)=\int_0^{y^*}\left(2e^{-\Phi(0)x}-1\right)W(x-y)\, \mathrm{d}x \quad \text{for all $y \geq 0$}.\label{specnegV}
\end{equation}
\end{cor}

\begin{proof} We have from Theorem \ref{thm_main} that $V(y)=V(y,y^*)$, where we define

\[ V(y,x)=\mathbb{E}\left[\int_0^{\sigma(y,x)} (2F(Y_t^y)-1) \, \mathrm{d}t\right]. \]
Now

\begin{eqnarray*}
V(y,y^*)&=&\mathbb{E}\left[\int_0^\infty (2F(Y_t^y)-1) \mathbf{1}_{\{t<\sigma(y,y^*)\}} \, \mathrm{d}t\right]\\
&=&\int_0^\infty \int_{[0,{y^*}]}(2F(x)-1)\mathbb{P}(Y_t^y\in \mathrm{d}x,t<\sigma(y,y^*)) \, \mathrm{d}t\\
&=&\int_{[0,{y^*}]}(2F(x)-1)\int_0^\infty \mathbb{P}(Y_t^y\in \mathrm{d}x,t<\sigma(y,y^*)) \, \mathrm{d}t\\
&=&\int_0^{y^*}(2F(x)-1)u_{y^*}(y,x)\mathrm{d}x-U_{y^*}(y,\{0\}).
\end{eqnarray*}
Plugging in the result from Lemma \ref{lem_u} yields

\begin{eqnarray*}
V(y,y^*) & = & \int_0^{y^*}\left(1-2e^{-\Phi(0)x}\right)\left(W(y^*-y)\frac{W'(x)}{W'(y^*)}-W(x-y)\right) \, \mathrm{d}x-\frac{W(y^*-y)W(0)}{W'(y^*)} \\
 & = & \int_0^{y^*}\left(2e^{-\Phi(0)x}-1\right) W(x-y) \, \mathrm{d}x \\
 & & \quad + \frac{W(y^*-y)}{W'(y^*)} \left( \int_0^{y^*}\left(1-2e^{-\Phi(0)x} \right) W'(x) \, \mathrm{d}x -W(0) \right).
\end{eqnarray*}
If $X$ is of bounded variation, i.e. irregular upwards, continuity of $V$ requires in particular $V(y^*-,y^*)=V(y^*,y^*)=0$, which readily implies that $y^*$ solves (\ref{specnegy}) and that (\ref{specnegV}) holds, since $W(x)=0$ for $x<0$ and $W(0)>0$. If $X$ is of unbounded variation, i.e. regular upwards, the smooth pasting condition at $y^*$ (cf. Theorem \ref{thm_main} (i)) again readily implies that $y^*$ solves (\ref{specnegy}) and (\ref{specnegV}) holds, since in this case $W(x)=0$ for $x<0$, $W(0)=0$ and $W'(0+)>0$.

Finally it remains to show that (\ref{specnegy}) has at most one solution. This is straightforward since the function $g$ defined by

\[g(y)=\int_0^{y}(1-2e^{-\Phi(0)x})W'(x)\,\mathrm{d}x \]
satisfies $g(0)=0$, $g'(y)<0$ for $y\in(0,\mathrm{m})$ (here $\mathrm{m}=\log(2)/\Phi(0)$) and $g'(y)>0$ for $y>\mathrm{m}$.
\end{proof}

We conclude with some explicit examples.

\begin{example} Let $X$ be a Brownian motion with drift, i.e. $X_t = \sigma B_t + \mu t$ for $\sigma>0$ and $\mu<0$ where $B$ is a standard Brownian motion. The Laplace exponent $\psi$ is given by

\[ \psi(z)= \frac{\sigma^2}{2} z^2 + \mu z. \]
It is straightforward to check that $\Phi(0)=-2\mu/\sigma^2$ and

\[ W(x) = \mu \left( 1-e^{-2\mu x/\sigma^2} \right) \quad \text{for $x>0$.} \]
Plugging this into Corollary \ref{cor} shows that $y^*$ is the unique solution of the equation in $y$

\[ e^{-2\mu y/\sigma^2}-1+\frac{4\mu}{\sigma^2}y=0 \]
and

\[ V(y) = \left( 2\mu y - \frac{3\sigma^2}{2} \right) e^{2\mu y/\sigma^2} -\mu(y^*-y) + \sigma^2 e^{2\mu y^*/\sigma^2} + \frac{\sigma^2}{2} \quad \text{for $y \in [0,y^*]$.} \]
\end{example}

\begin{example}\label{exampp} In this example we consider the jump-diffusion $X_t=\sigma B_t + \mu t - \sum_{i=1}^{N_t} Y_i$, where $B$ is a standard Brownian motion, $N$ is a Poisson process with intensity $\lambda>0$, $(Y_i)_{i \geq 1}$ is a sequence of i.i.d. exponentially distributed random variables with parameter $\eta>0$, and where$\sigma>0$ and $\mu \in \mathbb{R}$. The Laplace exponent $\psi$ of $X$ is given by

\[ \psi(z)= \frac{\sigma^2}{2} z^2 + \mu z - \frac{\lambda z}{\eta+z}. \]
Choosing the parameters such that $\psi'(0)<0$ we see that $\psi$ has roots $\beta_1<-\eta$, $\beta_2=0$ and $\beta_3>0$, with

\begin{multline*}
\beta_{1} = - \left( \frac{\eta}{2}+\frac{\mu}{\sigma^2} \right) - \sqrt{\left( \frac{\eta}{2}+\frac{\mu}{\sigma^2} \right)^2-2 \left( \frac{\mu \eta-\lambda}{\sigma^2} \right)} \quad \text{and} \\
\beta_{3} = - \left( \frac{\eta}{2}+\frac{\mu}{\sigma^2} \right) + \sqrt{\left( \frac{\eta}{2}+\frac{\mu}{\sigma^2} \right)^2-2 \left( \frac{\mu \eta-\lambda}{\sigma^2} \right)}.
\end{multline*}
Furthermore

\[ W(x) = C_1 e^{\beta_1 x} + C_2 + C_3 e^{\beta_3 x} \quad \text{for $x > 0$} \]
where

\[ C_1=\frac{2(\eta+\beta_1)}{\sigma^2 \beta_1 (\beta_1-\beta_3)}, \quad C_2=\frac{2\eta}{\sigma^2 \beta_1 \beta_3} \quad \text{and} \quad C_3=\frac{2(\eta+\beta_3)}{\sigma^2 \beta_3 (\beta_3-\beta_1)} \]
as follows directly from the definition (see also \cite{Baurdoux09}). Plugging this into (\ref{specnegy}) and (\ref{specnegV}) (together with $\Phi(0)=\beta_3$) leads to Figure 1. Note that Figure 1 also illustrates smooth pasting at $y=y^*$ (cf. Theorem \ref{thm_main} (i)).

For comparison, if we set $\sigma=0$ and take $\mu \in (0,\lambda/\eta)$ the resulting process $X_t=\mu t - \sum_{i=1}^{N_t} Y_i$ is irregular downwards. In this case we have  $\Phi(0)=\lambda/\mu-\eta$ and

\[ W(x) = \frac{\lambda}{\mu(\lambda-\mu \eta)} e^{(\lambda/\mu-\eta)x}-\frac{\eta}{\lambda-\mu \eta} \quad \text{for $x > 0$.} \]
This setting is illustrated in Figure 2. Note that indeed the plot illustrates there is no smooth pasting at $y=y^*$ (cf. Theorem \ref{thm_main} (ii)).
\end{example}

\begin{figure}[!ht]
\begin{center}
\includegraphics[width=12cm]{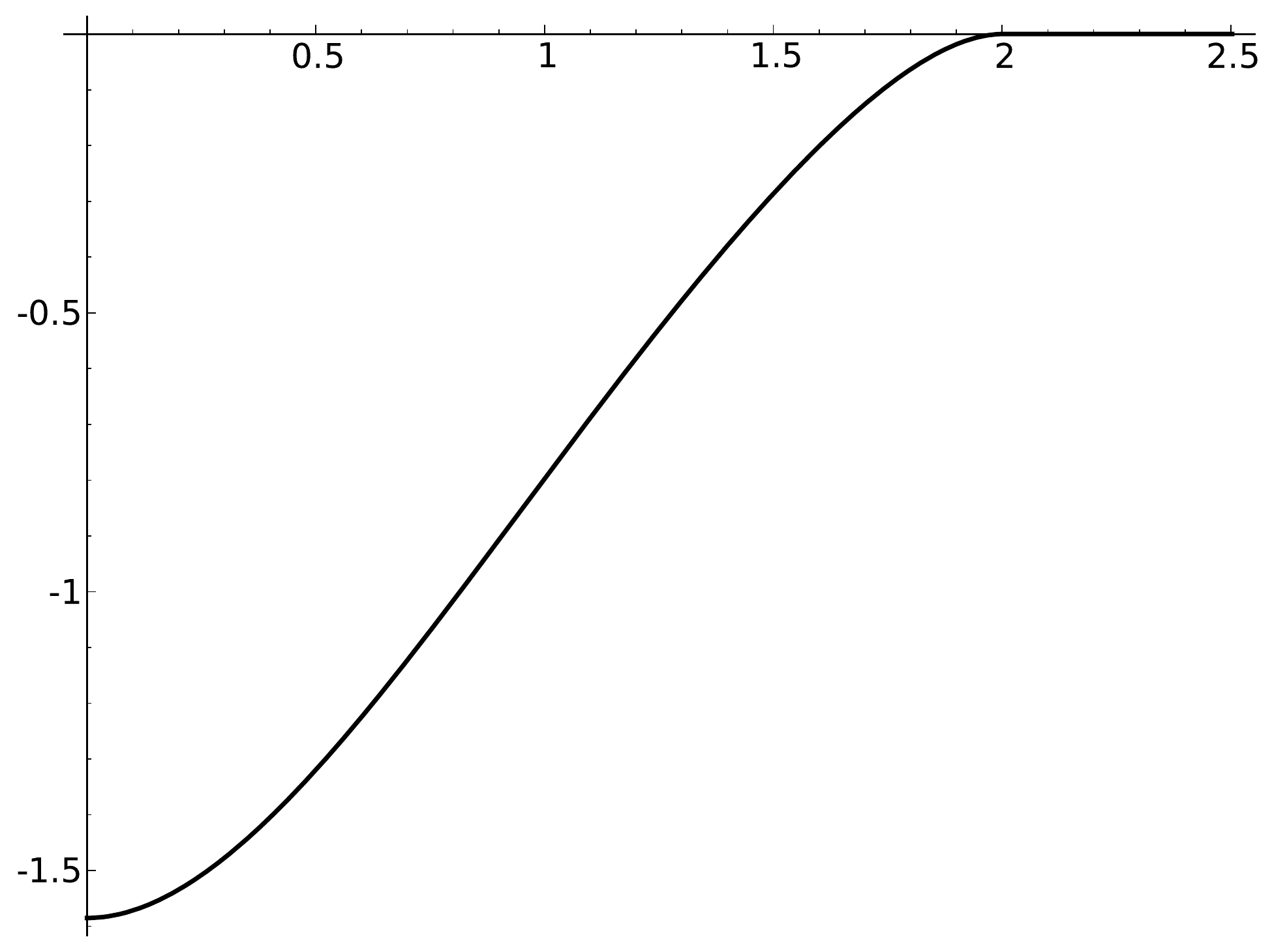}
\caption{A plot of the value function $V$ in the setting of Example \ref{exampp}, with $\sigma=\mu=1/2$ and $\lambda=\eta=1$. Note that $y^* \approx 2.0$. Regularity downwards leads to smooth pasting at $y^*$}
\end{center}
\end{figure}

\begin{figure}[!ht]
\begin{center}
\includegraphics[width=12cm]{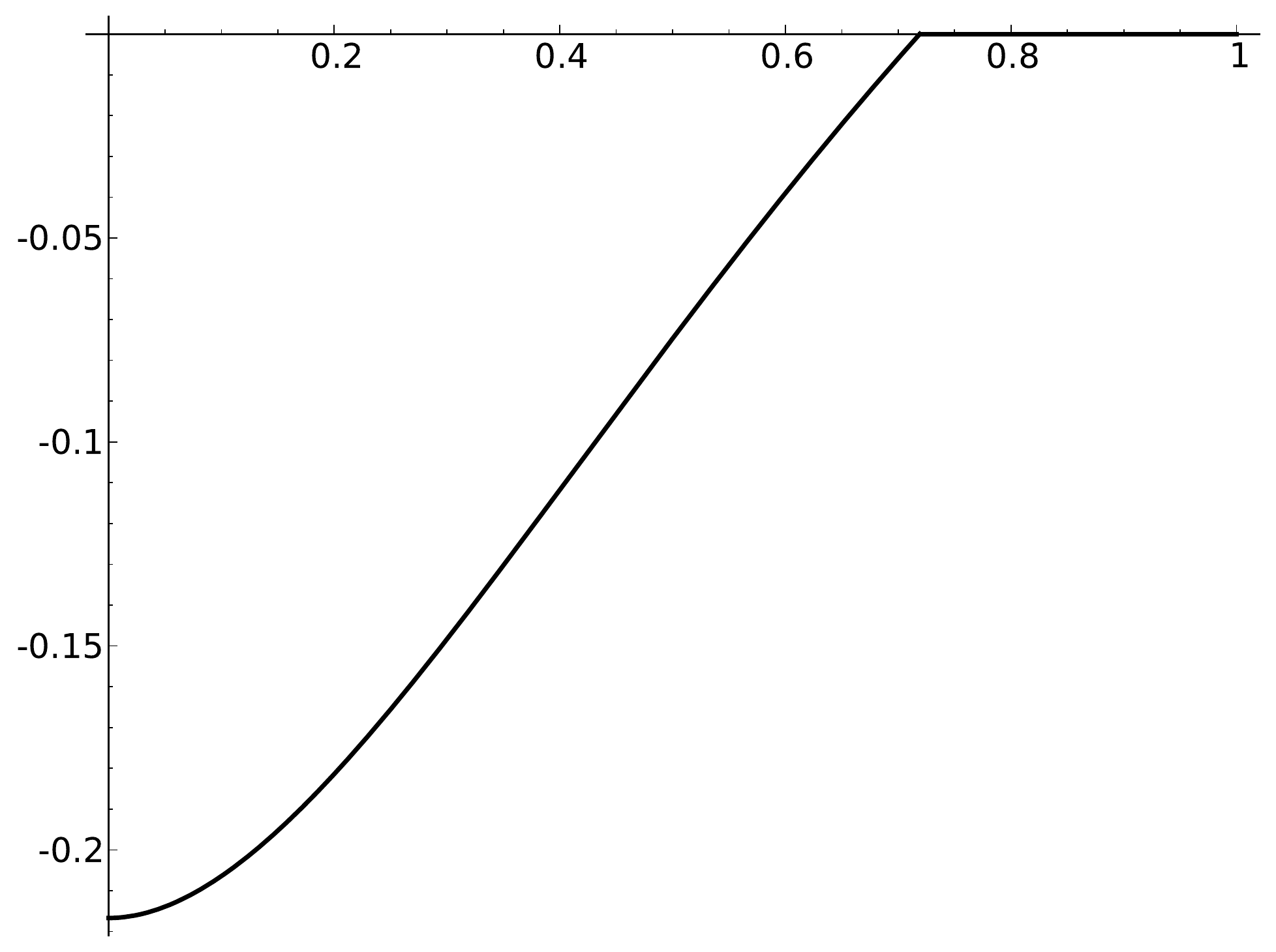}
\caption{A plot of the value function $V$ in the setting of Example \ref{exampp}, with $\sigma=0$, $\mu=2$, $\lambda=5$ and $\eta=1/5$. Note that $y^* \approx 0.73$. Irregularity downwards implies there is no smooth pasting at $y^*$}
\end{center}
\end{figure}

\clearpage

\section*{Acknowledgements} The authors are very grateful to Jenny Sexton for useful suggestions and discussion as well as to two anonymous referees for the helpful comments.

\section*{Appendix}
\renewcommand{\theequation}{A.\arabic{equation}}
\setcounter{equation}{0}

Here we prove the four lemmas from Section \ref{sec_prel}.

\begin{proof}[Proof of Lemma \ref{Ffacts}] For any $q>0$ let $\mathrm{e}(q)$ be a random variable independent of $X$ following an exponential distribution with mean $1/q$. From the Wiener--Hopf factorisation (see in particular part (ii) and (iii) of Theorem 6.16 in \cite{Kyprianou06} e.g.) we know that for any $\beta>0$

\[ \mathbb{E}\left[ e^{-\beta \overline{X}_{\mathrm{e}(q)}} \right] = \exp \left( \int_0^{\infty} \int_{(0,\infty)} \frac{1}{t} e^{-qt} \left( e^{-\beta x}-1 \right) \mathbb{P}(X_t \in \mathrm{d}x) \, \mathrm{d}t \right). \]
Then

\begin{multline*}
\mathbb{P}(\Xmax=0) = \lim_{\beta \to \infty} \mathbb{E}\left[ e^{-\beta \Xmax} \right] = \lim_{\beta \to \infty} \lim_{q \downarrow 0} \mathbb{E}\left[ e^{-\beta \overline{X}_{\mathrm{e}(q)}} \right] \\
= \exp \left( - \int_0^{\infty} \int_{(0,\infty)} \frac{1}{t} \mathbb{P}(X_t \in \mathrm{d}x) \, \mathrm{d}t \right) = \exp \left( - \int_0^{\infty} \frac{1}{t} \mathbb{P}(X_t >0) \, \mathrm{d}t \right),
\end{multline*}
hence $\mathbb{P}(\Xmax=0)>0$ if and only if

\begin{equation}\label{20mei1}
\int_0^{1} \frac{1}{t} \mathbb{P}(X_t >0) \, \mathrm{d}t<\infty \quad \text{and} \quad \int_1^{\infty} \frac{1}{t} \mathbb{P}(X_t >0) \, \mathrm{d}t<\infty.
\end{equation}
Indeed, the second integral is finite since we assumed (\ref{18mei1}) while the first integral is finite precisely when $X$ is irregular upwards (cf. \cite{Kyprianou06} Theorem 6.5).
\end{proof}

\begin{proof}[Proof of Lemma \ref{lem_Fcont}] From Wiener--Hopf theory we know that $\Xmax$ is equal in law to $H_{\mathrm{e}}$, where $H$ is the ascending ladder height process and $\mathrm{e}$ is an exponentially distributed random variable independent of $X$ and with parameter $\kappa(0,0)$, where $\kappa$ denotes the Laplace exponent of the ladder process (see e.g. Chapter 6 in \cite{Kyprianou06} for further details). The Laplace exponent $\psi$ of $H$ given by

\[ \psi(z) = - \frac{1}{t} \log \mathbb{E} \left[ e^{-zH_t} \right] \quad \text{for all $t>0$} \]
can be expressed as

\[ \psi(z) = d_H z + \int_{\mathbb{R}_{>0}} \left( 1-e^{-zx} \right) \Pi_H(\mathrm{d}x), \]
where $d_H \geq 0$ is the drift and the L\'evy measure $\Pi_H$ satisfies $\int_{\mathbb{R}_{>0}} (1 \wedge x) \Pi_H(\mathrm{d}x)<\infty$.

From \cite{Bertoin96} Proposition 17 on p. 172 and Theorem 19 on p. 175 it follows that $X$ creeping upwards is equivalent to $d_H>0$, and if this is the case $F$ has a bounded, continuous and positive density on $\mathbb{R}_{>0}$.

Henceforth suppose $d_H=0$, while it remains to show that $F$ is continuous on $\mathbb{R}_{\geq 0}$ when $X$ is not a compound Poisson process. If $\Pi_H(\mathbb{R}_{>0})=\infty$ it is not difficult to see that $F$ is continuous, cf. Theorem 5.4 (i) in \cite{Kyprianou06}. Let now $\Pi_H(\mathbb{R}_{>0})<\infty$. Then $H$ is a compound Poisson process (with jump distribution $\Pi_H$ times a constant). In this case $\mathbb{P}(\Xmax=0)=\mathbb{P}(H_{\mathrm{e}}=0)>0$ and hence in particular $X$ is irregular upwards by the above Lemma \ref{Ffacts}. Denote by $\widehat{H}$ the ladder height process of $-X$ which is a subordinator (without killing as $X$ drifts to $-\infty$). Note that $\widehat{H}$ cannot be a compound Poisson process, because if it were then by the same argument as above $X$ would be irregular downwards in addition to irregular upwards and hence $X$ would be compound Poisson which we excluded. So either $d_{\widehat{H}}>0$ or  $\Pi_{\widehat{H}}(\mathbb{R}_{>0})=\infty$ (or both) must hold, which in turn implies (analogue to above, cf. Theorem 5.4 (i) in \cite{Kyprianou06}) that the renewal measure $\widehat{U}$ of $\widehat{H}$ given by

\[ \widehat{U}(\mathrm{d}x) = \int_0^{\infty} \mathbb{P}(\widehat{H}_t \in \mathrm{d}x) \, \mathrm{d}t \]
has no atoms. The so-called equation amicale invers\'ee from \cite{Vigon02} reads

\[ \Pi_H((y,\infty)) = \int_{\mathbb{R}_{\geq 0}} \Pi((x+y,\infty)) \widehat{U}(\mathrm{d}x) \]
and shows that $\Pi_H$ has no atoms since $\widehat{U}$ has no atoms. Hence, $H$ is compound Poisson with a continuous jump distribution. It is now a straightforward exercise to find an expression for the law of $H_{\mathrm{e}}$ (by conditioning on the number of jumps $H$ experiences before $\mathrm{e}$) and to deduce that the continuous jump distribution guarantees that the distribution function of $H_{\mathrm{e}}$ is continuous on $\mathbb{R}_{\geq 0}$.
\end{proof}

\begin{proof}[Proof of Lemma \ref{limsuph}]
From p.10 (see the display in the middle of that page) in \cite{Chaumont05} we know that
\[\lim_{\eps\downarrow 0}\frac{\mathbb{P}(\tau^+(c-\eps)<\tau^-(-\eps))}{h(\eps)}=n(\overline{\nu}>c),\]
where $n$ denotes the excursion measure, $\overline{\nu}$ the height of a generic excursion $\nu$ and $h$ is the renewal function of the downward ladder height process $\widehat{H}$, i.e.
\[
h(x)=\int_0^\infty \mathbb{P}(\widehat{H}_t\geq -x)\,\mathrm{d}t.\]
 This function $h$ is subadditive (as is easily seen using the Markov property) and satisfies $h(0+)=0$ when $X$ is regular downwards. As $h$ is not a constant function, it also holds that
\[\limsup_{\eps\downarrow 0} \frac{h(\eps)}{\eps}>0.\]
Indeed, if this limsup were $0$, then $h'_+(0)=0$ which would imply that $h$ would be the zero function since $h$ is non-decreasing and for any $y\geq 0$
\[\frac{h(y+\varepsilon)-h(y)}{\varepsilon}\leq \frac{h(\varepsilon)}{\varepsilon}.\]
This concludes the proof of Lemma \ref{limsuph}.
\end{proof}

\begin{proof}[Proof of Lemma \ref{lemmm}]
Denote $T_-(x)=\inf \{ t \geq 0 \, | \, X_t \leq x \}$. We break the proof up in three steps, by considering the cases where $X$ is regular dowards, a subordinator and irregular downwards, respectively.

Step 1. Let us assume that $X$ is regular downwards. Note that we may equivalently write $f(x) = \inf_{\tau} \mathbb{E}_x [ Z_\tau ]$ where $Z_t = a t + \mathbf{1}_{\{ X_t \geq 0 \}} b$ for $t \geq 0$. For any $n \in \mathbb{N}$ let $b^{(n)}$ be a continuous non-increasing function so that $b^{(n)}(x)=0$ for $x \leq -1/n$ and $b^{(n)}(x)=b$ for $x \geq 0$. Define the process $Z^{(n)}$ as

\[ Z^{(n)}_t = a (t \wedge T_-(-n)) + b^{(n)}(X_{t \wedge T_-(-n)}) \quad \text{for $t \geq 0$}\]
and define $f^{(n)}(x) = \inf_{\tau} \mathbb{E}_x [ Z^{(n)}_\tau ]$. Take some $x \in \mathbb{R}$. As $n \to \infty$ we have for any stopping time $\tau$

\[  a (\tau \wedge T_-(-n)) \uparrow a\tau \quad \text{and} \quad b^{(n)}(X_{\tau \wedge T_-(-n)}) \to \mathbf{1}_{\{ X_\tau \geq 0 \}} b \quad \text{$\mathbb{P}_x$-a.s.} \]
so that by monotone and dominated convergence $ \mathbb{E}_x [ Z^{(n)}_\tau ] \to \mathbb{E}_x \left[ Z_\tau \right]$. It is readily checked that this implies

\begin{equation}\label{lemmm1}
f^{(n)}(x) \to f(x).
\end{equation}
Fix some $n \in \mathbb{N}$. It is straightforward to check that $f^{(n)}$ is continuous. Indeed for any $x_1<x_2$ we have
\begin{eqnarray}
0 \leq f^{(n)}(x_2)-f^{(n)}(x_1) &=&\inf_\tau \mathbb{E}_0[a (\tau \wedge T_-(x_2-n))+b^{(n)}(x_2+ X_{\tau \wedge T_-(x_2-n)})]\nonumber\\
&&-\inf_\tau \mathbb{E}_0[a (\tau \wedge T_-(x_1-n))+b^{(n)}(x_1+ X_{\tau \wedge T_-(x_1-n)})]\nonumber\\
&\leq&\mathbb{E}_0[a (\hat{\tau} \wedge T_-(x_2-n))+b^{(n)}(x_2+ X_{\hat{\tau} \wedge T_-(x_2-n)})]\nonumber\\
&&-\mathbb{E}_0[a (\hat{\tau} \wedge T_-(x_1-n))+b^{(n)}(x_1+ X_{\hat{\tau} \wedge T_-(x_1-n)})],\label{ultirhs}
\end{eqnarray}
where $\hat{\tau}$ is an optimal stopping time for the optimal stopping problem 
\[\inf_\tau\mathbb{E}_0[a (\tau \wedge T_-(x_1-n))+b^{(n)}(x_1+ X_{\tau \wedge T_-(x_1-n)})].\] Now, note that the right-hand side (\ref{ultirhs}) vanishes as $x_2-x_1 \downarrow 0$ on account of the assumption that $X$ is regular downwards. (In general the infimum in $f^{(n)}(x_1)$ might not be attained, however in such a case the above argument still applies using a sequence of stopping times instead of $\hat{\tau}$). Furthermore, $Z^{(n)}$ can be written as a continuous function of the Markov process $(t \wedge T_-(-n),X_{t \wedge T_-(-n)})_{t \geq 0}$ and $\sup_{t \geq 0}|Z^{(n)}_t|$ is integrable since $T_-(-n)$ has finite mean due to the fact that $X$ drifts to $-\infty$ (see for example \cite{Bertoin96} Proposition 17 on p. 172). These facts ensure that the general theory of optimal stopping applies and thus an optimal stopping time $\tau^*$ for $f^{(n)}$ is given by

\begin{equation}\label{lemmm2}
\tau^* = \inf \{ t \geq 0 \, | \, X_t \in \mathcal{S}^{(n)} \}, \text{ where } \mathcal{S}^{(n)}=\{ x \, | \, f^{(n)}(x)=b^{(n)}(x) \},
\end{equation}
cf. \cite{Peskir06} Corrolary 2.9 on p. 46. Indeed, using the (lower) Snell envelope $\hat{Z}^{(n)}$ of $Z^{(n)}$ yields for any $t \geq 0$

\begin{multline*}
\hat{Z}^{(n)}_t = \essinf{\tau \geq t} \mathbb{E}_x \left[ \left. Z^{(n)}_\tau \, \right| \mathcal{F}_t \right] = \essinf{\tau \geq t} \mathbb{E}_x \left[ \left. (\mathbf{1}_{\{T_-(-n) \leq t\}} + \mathbf{1}_{\{T_-(-n) > t\}}) Z^{(n)}_\tau \, \right| \mathcal{F}_t \right] \\
= \mathbf{1}_{\{T_-(-n) \leq t\}} \left( aT_-(-n) + b^{(n)}(X_{T_-(-n)}) \right) + \mathbf{1}_{\{T_-(-n) > t\}} \left( at + f^{(n)}(X_t) \right).
\end{multline*}
Since

\[ Z^{(n)}_t = \mathbf{1}_{\{T_-(-n) \leq t\}} \left( aT_-(-n) + b^{(n)}(X_{T_-(-n)}) \right) + \mathbf{1}_{\{T_-(-n) > t\}} \left( at + b^{(n)}(X_t) \right), \]
we retrieve that the optimal stopping time for $f^{(n)}$ as dictated by the general theory of optimal stopping, namely $\inf \{ t \geq 0 \, | \, \hat{Z}^{(n)}_t=Z^{(n)}_t \}$ coincides with $\tau^*$.

Now, as $b^{(n)}(x)=0$ for $x \leq -1/n$ and $f^{(n)}$ is non-increasing it follows that

\[  \mathcal{S}^{(n)} \cap (-\infty,-1/n] = \{ x \in (-\infty,-1/n] \, | \, f^{(n)}(x)=0 \} = (-\infty, \bar{x}^{(n)}] \]
for some $\bar{x}^{(n)} \in [-n,-1/n]$. Let $x_0=\liminf_{n \to \infty} \bar{x}^{(n)}$. Then the result follows on account of (\ref{lemmm1}) provided we show $x_0>-\infty$. For this, suppose we had $x_0=-\infty$. Along a subsequence such that $\bar{x}^{(n')} \to -\infty$ we would have

\begin{multline}\label{Kevin2}
f^{(n')}(-n'/2) = \mathbb{E}_{-n'/2} \left[ Z^{(n')}_{\tau^*} \right] \\
= \mathbb{E}_{-n'/2} \left[ a(\tau^* \wedge T_-(-n')) \right] + \mathbb{E}_{-n'/2} \left[ b^{(n')}(X_{\tau^* \wedge T_-(-n')}) \right] \to \infty
\end{multline}
since the second expectation above remains bounded while for the first expectation we have

\[  \mathbb{E}_{-n'/2} \left[ a(\tau^* \wedge T_-(-n')) \right] \geq  \mathbb{E}_{-n'/2} \left[ a(T_+(-1/n') \wedge T_-(\bar{x}^{(n')})) \right] \to \infty. \]
However (\ref{Kevin2}) is impossible since $f^{(n)}(x) \leq 0$ for all $n, x$, and hence this step is done.

Step 2. As an intermediate step, let us now assume that $X$ is a subordinator, i.e. $X$ is a L\'evy process with non-decreasing paths. We also assume the drift of $X$ is strictly positive. As a consequence, $X$ drifts to $\infty$, is regular upwards and $\mathbb{E}_x[T_+(0)]<\infty$ for all $x$. Note that we may write

\[ f(x) = \inf_{\tau} \mathbb{E}_x \left[ a (\tau \wedge T_+(0)) + \mathbf{1}_{\{ \tau \geq T_+(0) \}} b \right]. \]
For any $n \in \mathbb{N}$, let $b^{(n)}$ be defined as in Step 1 and define the process $Z^{(n)}$ as

\[ Z^{(n)}_t = a (t \wedge T_+(0)) + b^{(n)}(X_{t \wedge T_+(0)}) \quad \text{for $t \geq 0$.} \]
Furthermore set $f^{(n)}(x) = \inf_{\tau} \mathbb{E}_x [ Z^{(n)}_\tau ]$. Note we again have that $f^{(n)}(x) \to f(x)$ for all $x$ as $n \to \infty$. Fix some $n$. We have that $\sup_{t \geq 0} |Z^{(n)}_t|$ is integrable since $T_+(0)$ has finite mean, $Z^{(n)}$ is a continuous function of the Markov process $(t \wedge T_+(0),X_{t \wedge T_+(0)})_{t \geq 0}$ and $f^{(n)}$ is continuous which can be seen by the same argument as in Step 1 (thereby using that $X$ is regular upwards). The same argument as in Step 1 shows that the optimal stopping time $\tau^*$ for $f^{(n)}$ is again given by (\ref{lemmm2}).

Fix some $x<0$. For all $n$ such that $x<-1/n$ we have either $f^{(n)}(x)=0$ (or, equivalently, $\tau^*=0$) or $f^{(n)}(x)<0$. In the latter case, the monotonicity of the paths of $X$ and of the function $f^{(n)}$ imply $\tau^* \geq T_+(-1/n)$. It follows that if $f^{(n)}(x)<0$ we have

\begin{equation}\label{Kevin3}
f^{(n)}(x) = \mathbb{E}_x \left[ a (\tau^* \wedge T_+(0)) + \mathbf{1}_{\{ \tau^* \geq T_+(0) \}} b \right] \geq a \mathbb{E}_x \left[ T_+(-1/n) \right] +b.
\end{equation}
Take some (small) $\eps>0$. Since $x \mapsto \mathbb{E}_x[T_+(-\eps)]$ is decreasing and tends to $\infty$ as $x \to -\infty$ there exists $x_0$ such that for all $x<x_0$ we have $\mathbb{E}_x[T_+(-\eps)] \geq -b/a$. We will show that for all $x$ satisfying this property we have $f(x)=0$. Indeed, take such $x$ and suppose we had $f(x)<0$. Then a subsequence $n'$ exists along which $f^{(n')}(x)<0$. However (\ref{Kevin3}) implies that for all $n'$ such that $1/n' < \eps$ we have $f^{(n')}(x) \geq a \mathbb{E}_x \left[ T_+(-\eps) \right] +b \geq 0$ which is a contradiction. Hence indeed $f(x)=0$ and we conclude that the result of the lemma also holds when $X$ is a subordinator.

Step 3. Finally suppose that $X$ is irregular downwards -- hence complementing Step 1. Then $X$ is necessarily of bounded variation and thus consists of a drift term plus the sum of its jumps (see equation (2.22) in \cite{Kyprianou06}). Hence in particular a subordinator $\widehat{X}$ exists with strictly positive drift such that $\widehat{X}_t \geq X_t$ a.s. for all $t \geq 0$. From Step 2 above we know that the result holds for $\widehat{X}$ and it readily follows this means the result also holds for $X$.
\end{proof}

\end{document}